\documentclass{amsart}
\usepackage{amscd,amsmath,amssymb,amsthm,amsfonts,epsfig,graphics}

\theoremstyle{theorem}

\theoremstyle{definition}
\newtheorem{definition}{Definition}
\newtheorem{prop}{Proposition}
\newtheorem{remark}{Remark}

\title{Hemi-slant submanifolds in metallic Riemannian manifolds}
\author{Cristina E. Hretcanu and Adara M. Blaga}

     \keywords{Metallic Riemannian structure, Golden Riemannian structure, almost product structure, induced structure on submanifold, slant submanifold, hemi-slant submanifold.}
     \subjclass[2010]{53B20, 53B25, 53C15}

\begin{document}

\normalfont
 \begin{abstract}
     The aim of our paper is to focus on some properties of hemi-slant submanifolds in metallic Riemannian manifolds. We give some characterizations for submanifolds to be hemi-slant submanifolds in metallic or Golden Riemannian manifolds and we obtain integrability conditions for the distributions involved. Examples of hemi-slant submanifolds in metallic and Golden Riemannian manifolds are given.
     \end{abstract}
     \maketitle

   \section{Introduction} \label{intro}

   The notion of metallic structure (and, in particular, Golden structure) on a Riemannian manifold was initially studied in (\cite{CrHr},\cite{Hr2},\cite{Hr3},\cite{Hr4}).
   In (\cite{Hr2}), the authors of the present paper studied the properties of the slant and semi-slant submanifolds in metallic or Golden Riemannian manifolds and obtained some integrability conditions for the distributions involved in the semi-slant submanifolds of Riemannian manifolds endowed with metallic or Golden Riemannian structures.

   The geometry of slant submanifolds in complex manifolds, studied by B.Y. Chen in (\cite{Chen3},\cite{Chen4}) in the early 1990’s, was extended to semi-slant submanifold, pseudo-slant submanifold and bi-slant submanifold, respectively, in different types of differentiable manifolds. Semi-slant submanifolds in almost Hermitian manifolds were introduced by N. Papagiuc (\cite{Papaghiuc}). Semi-slant submanifolds in Sasakian manifolds were studies by J.L. Cabrerizo \textit{et al.} in (\cite{Cabrerizo1},\cite{Cabrerizo2}). A. Cariazzo \textit{et al.} (\cite{Carriazo}) studied bi-slant immersion in almost Hermitian manifolds and pseudo-slant submanifold in almost Hermitian manifolds. Slant and semi-slant submanifolds in almost product Riemannian manifolds were studied in (\cite{Atceken1},\cite{Li&Liu},\cite{Sahin}). The pseudo-slant submanifolds (also called hemi-slant submanifolds) in Kenmotsu or nearly Kenmotsu manifolds (\cite{Atceken3},\cite{Atceken4}), in LCS-manifolds (\cite{Atceken5}) or in locally decomposable Riemannian manifolds (\cite{Atceken6}) were studied by M. At\c{c}eken \textit{et al.} Properties of hemi-slant submanifolds in locally product Riemannian manifolds were studied by H.M. Ta\c{s}tan and F. Ozdem in \cite{Tastan}.

   The purpose of the present paper is to investigate the properties of hemi-slant submanifolds in metallic (or Golden) Riemannian manifolds. We find some integrability conditions for the distributions which are involved in such types of submanifolds in metallic and Golden Riemannian manifolds and we give some examples of hemi-slant submanifolds in metallic (or Golden) Riemannian manifolds.

Using a polynomial structure on a manifold (\cite{Goldberg1},\cite{Goldberg2}) and the metallic numbers (\cite{Spinadel}), we defined the metallic structure $J$ (\cite{Hr4}).

The name of this structure is provided by the metallic number $\sigma _{p,q}=\frac{p+\sqrt{p^{2}+4q}}{2}$ (i.e. the positive solution of the equation $x^{2}-px-q=0$) for positive integer values of $p$ and $q$.

If $\overline{M}$ is an $m$-dimensional manifold endowed with a tensor field $J$ of type $(1,1)$ such that:
\begin{equation}\label{e1}
J^{2}= pJ+qI,
\end{equation}
for $p$, $q\in\mathbb{N}^*$, where $I$ is the identity operator on the Lie algebra $\Gamma(T\overline{M})$, then the structure $J$ is a \textit{metallic structure}. In this situation, the pair $(\overline{M},J)$ is called \textit{metallic manifold}.

In particular, if $p=q=1$ one obtains the \textit{Golden structure} (\cite{CrHr}) determined by a $(1,1)$-tensor field $J$  which verifies $J^{2}= J + I$. In this case, $(\overline{M},J)$ is called {\it Golden manifold} (\cite{CrHr}).

If ($\overline{M}, \overline{g})$ is a Riemannian manifold endowed with a metallic (or a Golden) structure $J$, such that the Riemannian metric $\overline{g}$ is $J$-compatible, i.e.:
\begin{equation} \label{e2}
\overline{g}(JX, Y)= \overline{g}(X, JY),
\end{equation}
for any $X, Y \in \Gamma(T\overline{M})$, then $(\overline{g},J)$ is called a {\it metallic} (or a {\it Golden) Riemannian structure}  and $(\overline{M},\overline{g},J)$ is a {\it metallic (or a Golden) Riemannian manifold} (\cite{Hr4}).

Moreover, we have:
\begin{equation} \label{e3}
\overline{g}(JX, JY)=\overline{g}(J^{2}X, Y) =p \overline{g}(JX,Y)+q \overline{g}(X,Y),
\end{equation}
for any $X, Y \in \Gamma(T\overline{M})$ (\cite{Hr4}).
\normalfont

Any almost product structure $F$ on $\overline{M}$ induces two metallic structures on $\overline{M}$:
\begin{equation}\label{e4}
J= \frac{p}{2}I \pm \frac{2\sigma _{p, q}-p}{2}F,
\end{equation}
where $I$ is the identity operator on the Lie algebra $\Gamma(T\overline{M})$ (\cite{Hr4}).

\section{On the metallic Riemannian manifolds and its submanifolds}

Let $M$ be an $m'$-dimensional submanifold, isometrically immersed in the $m$-dimensional metallic (or Golden) Riemannian manifold ($\overline{M}, \overline{g},J)$ with $m, m' \in \mathbb{N}^{*}$ and $m > m'$.
Let $T_{x}M$  be the tangent space of $M$ in a point $x \in M$ and $T_{x}^{\bot }M$ the normal space of $M$ in $x$. The tangent space $T_x\overline{M}$ can be decomposed into the direct sum:
$$T_x\overline{M}=T_x M\oplus T_x^{\perp}M,$$ for any $x\in M$. Let $i_{*}$ be the differential of the immersion $i: M \rightarrow\overline{M}$.

The induced Riemannian metric $g$ on $M$ is given by $g(X, Y)=\overline{g}(i_{*}X, i_{*}Y)$, for any $X, Y \in \Gamma(TM)$. For the simplification of the notations, in the rest of the paper we shall note by $X$ the vector field $i_{*}X$, for any $X \in \Gamma(TM)$.

If we denote by $TX$ and $NX$, respectively, the tangential and normal parts of $JX$, for any $X \in \Gamma(TM)$, then we get:
\begin{equation}\label{e5}
JX = TX + NX,
\end{equation}
 $T:\Gamma(TM)\rightarrow \Gamma(TM)$, $ TX:=(J X)^T$ and $N:\Gamma(TM)\rightarrow \Gamma(T^{\perp}M)$, $NX:=(J X)^{\perp}$.

For any $V \in \Gamma(T^{\perp}M)$, the tangential and normal parts of $JV$ satisfy:
\begin{equation}\label{e6}
JV = tV + nV,
\end{equation}
  $t:\Gamma(T^{\perp}M)\rightarrow \Gamma(TM)$, $tV:=(J V)^{T}$ and $ n:\Gamma(T^{\perp}M)\rightarrow \Gamma(T^{\perp}M)$, $nV:=(J V)^{\perp}$.

We remark that the maps $T$ and $n$ are $\overline{g}$-symmetric (\cite{Blaga_Hr}):
\begin{equation}\label{e7}
(i)\: \overline{g}(TX,Y)=\overline{g}(X,TY), \quad (ii)\: \overline{g}(nU,V)=\overline{g}(U,nV),
\end{equation}
for any $X, Y\in \Gamma(TM)$ and $U, V \in \Gamma(T^{\perp}M)$. Moreover, we get
\begin{equation}\label{e8}
\overline{g}(NX,U)=\overline{g}(X,tU),
\end{equation}
for any $X\in \Gamma(TM)$ and $U\in \Gamma(T^{\perp}M)$.

By using (\ref{e5}), (\ref{e6}) and (\ref{e1}), we obtain:
\begin{remark}
 If $M$ is a submanifold in a metallic Riemannian manifold $(\overline{M}, \overline{g}, J)$, then:
\begin{equation} \label{e99}
(i) \: T^{2}X = pTX+qX-tNX, \quad (ii) \: pNX= NTX+nNX
\end{equation}
and
\begin{equation} \label{e100}
(i) \: n^{2}V =  pnV+qV-NtV, \quad (ii) \: ptV= TtV+tnV,
\end{equation}
for any $X \in \Gamma(TM)$ and $V \in \Gamma(T^{\bot}M)$.
For $p=q=1$ and $M$ is a submanifold in a Golden Riemannian manifold $(\overline{M}, \overline{g}, J)$ then, for any $X \in \Gamma(TM)$ we get
$ T^{2}X = TX+X-tNX$,  $NX= NTX+ nNX$ and for any $V \in \Gamma(T^{\bot}M)$ we get $n^{2}V = nV+V-NtV$,  $tV= TtV+tnV$.
\end{remark}

\begin{remark}(\cite{Hr6})
Let $(\overline{M}, \overline{g})$ be a Riemannian manifold endowed with an almost product structure $F$ and let $J$ be one of the two metallic structures induced by $F$ on $\overline{M}$. If $M$ is a submanifold in the almost product Riemannian manifold $(\overline{M}, \overline{g}, F)$, then:
\begin{equation} \label{e9}
(i) \: TX = \frac{p}{2}X \pm \frac{2\sigma-p}{2}fX, \quad (ii) \: NX= \pm \frac{2\sigma-p}{2}\omega X
\end{equation}
and
\begin{equation} \label{e10}
(i) \: tV =  \pm \frac{2\sigma-p}{2}BV, \quad (ii) \: nV= \frac{p}{2}V \pm \frac{2\sigma-p}{2}CV,
\end{equation}
for any $X \in \Gamma(TM)$ and $V \in \Gamma(T^{\bot}M)$, where
$FX = fX + \omega X$, $FV = BV + CV$, with $fX:=(F X)^{T}$, $\omega X:=(FX)^{\perp}$, $BV:=(F V)^T$ and $CV:=(F V)^{\perp}$.
\end{remark}

In the next considerations we denote by $\overline{\nabla}$ and $\nabla $ the Levi-Civita connections on $(\overline{M},\overline{g})$ and its submanifold $(M,g)$, respectively.

The Gauss and Weingarten formulas are given by:
\begin{equation}\label{e11}
(i) \: \overline{\nabla}_{X}Y=\nabla_{X}Y+h(X,Y), \quad  (ii) \: \overline{\nabla}_{X}V=-A_{V}X+\nabla_{X}^{\bot}V,
\end{equation}
for any $X, Y \in \Gamma(TM)$ and $V \in \Gamma(T^{\bot}M)$, where $h$ is the second fundamental form and $A_{V}$ is the shape operator. The second fundamental form $h$ and the shape operator $A_{V}$ are related by:
 \begin{equation}\label{e12}
 \overline{g}(h(X, Y),V)=\overline{g}(A_{V}X, Y).
 \end{equation}

\begin{definition}\label{d1}(\cite{Hr5})
If $(\overline{M},\overline{g}, J)$ is a metallic (or Golden) Riemannian manifold and $J$ is parallel with respect to the Levi-Civita connection $\overline{\nabla}$ on $\overline{M}$ (i.e. $\overline{\nabla}J=0$), we say that  $(\overline{M},\overline{g}, J)$ is a {\it locally metallic (or locally Golden) Riemannian manifold}.
\end{definition}

The covariant derivatives of the tangential and normal parts of $JX$ (and $JV$), $T$ and $N$ ($t$ and $n$, respectively) are given by (\cite{Hr5},\cite{Atceken3}):
\begin{equation}\label{e13}
(i) \: (\nabla_{X}T)Y=\nabla_{X}TY - T(\nabla_{X}Y), \quad (ii) \:(\overline{\nabla}_{X}N)Y=\nabla_{X}^{\bot}NY - N(\nabla_{X}Y),
\end{equation}
and
\begin{equation}\label{e14}
(i) \: (\nabla_{X}t)V=\nabla_{X}tV - t(\nabla_{X}^{\bot}V), \quad (ii) \:(\overline{\nabla}_{X}n)V=\nabla_{X}^{\bot}nV - n(\nabla_{X}^{\bot}V),
\end{equation}
for any $X$, $Y \in \Gamma(TM)$ and $V \in \Gamma(T^{\bot}M) $. From $\overline{g}(JX,Y)=\overline{g}(X,JY)$, it follows:
\begin{equation} \label{e15}
\overline{g}((\overline{\nabla}_XJ)Y,Z)=\overline{g}(Y,(\overline{\nabla}_XJ)Z),
\end{equation}
for any $X$, $Y$, $Z\in \Gamma(T\overline{M})$. Moreover, if $M$ is an isometrically immersed submanifold in the metallic Riemannian manifold $(\overline{M},\overline{g},J)$, then (\cite{Blaga2}):
\begin{equation}\label{e16}
\overline{g}((\nabla_X T)Y,Z)=\overline{g}(Y,(\nabla_X T)Z),
\end{equation}
for any $X$, $Y$, $Z\in \Gamma(TM)$.

\begin{prop}
If $M$ is a submanifold in a locally metallic (or Golden) Riemannian manifold $(\overline{M},\overline{g},J)$, then the covariant derivatives of $T$ and $N$ verify:
\begin{equation}\label{e17}
(i)  (\nabla_{X}T)Y=A_{NY}X+th(X,Y), \quad (ii) \: (\overline{\nabla}_{X}N)Y=nh(X,Y)-h(X,TY),
\end{equation}
and
\begin{equation}\label{e18}
(i)  (\nabla_{X}t)V=A_{nV}X - TA_{V}X, \quad (ii) \: (\overline{\nabla}_{X}n)V=-h(X,tV)-NA_{V}X,
\end{equation}
for any $X$, $Y \in \Gamma(TM)$ and $V \in \Gamma(T^{\bot}M)$.
\end{prop}

\begin{remark}
If $M$ is a totally geodesic submanifold in a locally metallic (or locally Golden) Riemannian manifold $(\overline{M},\overline{g},J)$, then:
$(\nabla_{X}T)Y=(\overline{\nabla}_{X}N)Y=(\nabla_{X}t)V=(\overline{\nabla}_{X}n)V=0,$
for any $X,Y \in \Gamma(TM)$ and $V \in \Gamma(T^{\perp}M)$.
\end{remark}

\begin{remark}
If $M$ is a submanifold in a locally metallic (or locally Golden) Riemannian manifold $(\overline{M},\overline{g},J)$, then we obtain:
\begin{equation}\label{e19}
\overline{g}((\overline{\nabla}_{X}N)Y,V )= \overline{g}((\nabla_{X}t)V,Y),
\end{equation}
for any $X$, $Y \in \Gamma(TM)$ and $V \in \Gamma(T^{\bot}M)$.
\end{remark}
\begin{proof}
From (\ref{e17})(ii) and (\ref{e7})(ii) we get
$$ \overline{g}((\overline{\nabla}_{X}N)Y,V )= \overline{g}(h(X,Y),nV )-\overline{g}(h(X,TY),V )= \overline{g}(A_{nV}X - TA_{V}X,Y) $$
and using (\ref{e18})(i) we obtain (\ref{e19}).
\end{proof}

\begin{prop}
Let $M$ be a submanifold in a locally metallic (or locally Golden) Riemannian manifold ($\overline{M},\overline{g},J$). Then $(\overline{\nabla}_{X}N)Y=0$ and $(\nabla_{X}t)V=0$, for any $X$, $Y \in \Gamma(TM)$, $V \in \Gamma(T^{\bot}M)$ if and only if the shape operator $A$ verifies:
\begin{equation}\label{e33}
 A_{nV}X=TA_{V}X=A_{V}TX.
\end{equation}
\end{prop}

\begin{proof}
From (\ref{e7})(ii) we get $\overline{g}(nh(X,Y),V)=\overline{g}(h(X,Y),nV)$, for any $X, Y \in \Gamma(TM)$, $V \in \Gamma(T^{\bot}M)$. Thus, we obtain:
$$\overline{g}((\overline{\nabla}_{X}N)Y,V)=\overline{g}(h(X,Y),nV)-\overline{g}(h(X,TY),V)=\overline{g}(A_{nV}X,Y)-\overline{g}(A_{V}X,TY),$$
for any $X, Y \in \Gamma(TM)$, $V \in \Gamma(T^{\bot}M)$. From (\ref{e17})(ii) and (\ref{e12}) we have
\begin{equation}\label{e34}
\overline{g}((\overline{\nabla}_{X}N)Y,V)=\overline{g}(A_{nV}X-TA_{V}X,Y)=\overline{g}(A_{nV}Y-A_{V}TY,X),
\end{equation}
for any $X, Y \in \Gamma(TM)$, $V \in \Gamma(T^{\bot}M)$. Thus, from  (\ref{e34}) and (\ref{e19}) we obtain the conclusion.
\end{proof}

\begin{prop}(\cite{Hr6})
If $M$ is a submanifold in a locally metallic (or locally Golden) Riemannian manifold $(\overline{M},\overline{g},J)$, then:
\begin{equation}\label{e20}
T([X,Y])=\nabla_{X}TY-\nabla_{Y}TX-A_{NY}X+A_{NX}Y
\end{equation}
and
\begin{equation}\label{e21}
N([X,Y])=h(X,TY)-h(TX,Y)+\nabla_{X}^{\bot}NY-\nabla_{Y}^{\bot}NX,
\end{equation}
for any $X, Y \in \Gamma(TM)$, where $\nabla$ is the Levi-Civita connection on $\Gamma(TM)$.
\end{prop}

\section{Hemi-slant submanifolds in metallic Riemannian manifolds}

In this section we recall the definition of a slant distribution and of a bi-slant submanifold in a metallic (or Golden) Riemannian manifold. Then, we define the hemi-slant submanifold and find some properties regarding the distributions involved in this type of submanifold, using a similar definition as for Riemannian product manifold (\cite{Tastan}).

\begin{definition}(\cite{Hr6})\label{d2}
Let $M$ be an immersed submanifold in a metallic (or Golden) Riemannian manifold $(\overline{M},\overline{g},J)$. A differentiable distribution $D$ on $M$ is called a {\it slant distribution} if the angle $\theta_{D}$ between $JX_{x}$ and the vector subspace $D_{x}$ is constant, for any $x \in M$ and any nonzero vector field $X_{x} \in \Gamma(D_{x})$. The constant angle $\theta_{D}$ is called the {\it slant angle} of the distribution $D$.
\end{definition}

\begin{prop}(\cite{Hr6})\label{pD}
Let $D$ be a differentiable distribution on a submanifold $M$ of a metallic (or Golden) Riemannian manifold $(\overline{M},\overline{g},J)$. The distribution $D$ is a slant distribution if and only if there exists a constant $\lambda \in [0, 1]$ such that:
\begin{equation}\label{e22}
(P_{D}T)^{2}X= \lambda(pP_{D}TX+qX),
\end{equation}
 for any $ X \in \Gamma(D)$, where $P_{D}$ is the orthogonal projection on $D$. Moreover, if $\theta_{D}$ is the slant angle of $D$, then it satisfies $\lambda = \cos^{2} \theta_{D}$.
 \end{prop}

\begin{definition}(\cite{Hr6})\label{d3}
Let $M$ be an immersed submanifold in a metallic (or Golden) Riemannian manifold $(\overline{M},\overline{g},J)$. We say that $M$ is a {\it bi-slant submanifold} of $\overline{M}$ if there exist two orthogonal differentiable distribution $D_{1}$ and $D_{2}$ on $M$ such that $TM = D_{1}\oplus D_{2}$, and $D_{1}$, $D_{2}$ are slant distributions with the slant angles $\theta_{1}$ and $\theta_{2}$, respectively. Moreover, $M$ is a {\it proper bi-slant submanifold} of $\overline{M}$ if $dim(D_{1})\cdot dim(D_{2}) \neq 0 $.
\end{definition}

A particular case of {\it bi-slant submanifold} in a metallic Riemannian manifold $(\overline{M},\overline{g},J)$ is the hemi-slant submanifold, defined in a similar manner as hemi-slant submanifold of the locally product Riemannian manifold (\cite{Tastan}):

\begin{definition}\label{d6}
An immersed submanifold $M$ in a metallic (or Golden) Riemannian manifold $(\overline{M},\overline{g},J)$ is a \textit{hemi-slant submanifold} if there exist two orthogonal distributions
$D^{\theta}$ and $D^{\perp}$ on $M$ such that:

(1) $TM$ admits the orthogonal direct decomposition $TM = D^{\theta}\oplus D^{\perp}$;

(2) The distribution $D^{\theta}$ is slant with angle $\theta \in [0,\frac{\pi}{2}]$;

(3) The distribution $D^{\perp}$ is anti-invariant distribution (i.e. $J(D^{\perp}) \subseteq \Gamma(T^{\perp}M$)).

Moreover, if $dim(D^{\theta})\cdot dim(D^{\perp}) \neq 0 $ and $\theta \in (0, \frac{\pi}{2})$, then $M$ is a proper hemi-slant submanifold.

\end{definition}

\begin{remark}
 If $M$ is a hemi-slant submanifold in a metallic Riemannian manifold $(\overline{M},\overline{g},J)$, with $TM = D^{\theta}\oplus D^{\perp}$, for particular cases we get:

- if $\theta=0$ and $dim(D^{\perp})=0$, then $M$ is an invariant submanifold;

- if $dim(D^{\theta})=0$ or $\theta = \frac{\pi}{2}$, then $M$ is an anti-invariant submanifold;

- if $dim(D^{\perp})=0$ and $\theta \neq 0$, then $M$ is a slant submanifold;

- if $dim(D^{\theta}) \cdot dim(D^{\perp}) \neq 0$ and $\theta=0$, then $M$ is a semi-invariant submanifold.

\end{remark}

In a similar manner as in (\cite{Hr6}, Proposition 10), we obtain:
\begin{remark}
 If $M$ is a hemi-slant submanifold in a metallic Riemannian manifold $(\overline{M},\overline{g},J)$, with $TM = D^{\theta}\oplus D^{\perp}$, then we get that $M$ is an anti-invariant submanifold if $\theta=\frac{\pi}{2}$ and $g(JX,Y)=0$, for any $X \in \Gamma(D^{\theta})$ and $X \in \Gamma(D^{\perp})$.
\end{remark}

 Let $M$ be a hemi-slant submanifold in a metallic Riemannian manifold $(\overline{M},\overline{g},J)$, with $TM = D^{\theta}\oplus D^{\perp}$ and let $P_{1}$ and $P_{2}$ be the orthogonal projections on $D^{\theta}$ and $D^{\perp}$, respectively. Thus, for any $X \in \Gamma(TM)$, we can consider the decomposition of $X=P_{1}X + P_{2}X$, where $P_{1}X \in \Gamma(D^{\theta})$ and $P_{2}X \in \Gamma(D^{\perp})$. From $J(D^{\perp}) \subseteq \Gamma(T^{\perp}M)$ we obtain:

\begin{prop}
 If $M$ is a hemi-slant submanifold in a metallic (or Golden) Riemannian manifold $(\overline{M},\overline{g},J)$, then:
\begin{equation}\label{e23}
JX= TP_{1}X+ NP_{1} X + NP_{2}X = TP_{1}X+ N X
\end{equation}
and
\begin{equation}\label{e24}
(i) JP_{2}X= NP_{2}X,  \: (ii) TP_{2}X=0,   \: (iii) TP_{1} X \in \Gamma(D^{\theta}),
\end{equation}
for any $X \in \Gamma(TM)$.
\end{prop}

\begin{remark}
 If $M$ is a hemi-slant submanifold in a metallic (or Golden) Riemannian manifold $(\overline{M},\overline{g},J)$, then:
\begin{equation}\label{e101}
 T^{\perp}M= N(D^{\theta}) \oplus N(D^{\perp}) \oplus \mu,
 \end{equation}
 where $\mu$ is an invariant subbundle of $T^{\perp}M$.
 \end{remark}
 \begin{proof}
For any $X \in \Gamma(D^{\theta})$ and $Z \in \Gamma(D^{\perp})$ we get
$$\overline{g}(NX,NZ)=\overline{g}(JX,JZ)=p \overline{g}(X,TZ)+q \overline{g}(X,Z)=0.$$
Thus, the distributions $N(D^{\theta})$ and $N(D^{\perp})$ are mutually perpendicular in $T^{\perp}M$. If we denote by $\mu$ the orthogonal complementary subbundle of $J(TM)$ in $T^{\perp}M$, then we obtain (\ref{e101}).
\end{proof}

\begin{remark}
 If $M$ is a hemi-slant submanifold in a metallic (or Golden) Riemannian manifold $(\overline{M},\overline{g},J)$, then:
$$\overline{g}(JP_{1}X,TP_{1}X)=\cos \theta(X) \| TP_{1}X\| \cdot \|JP_{1}X\|$$
and the cosine of the slant angle $\theta(X)=:\theta$ of the distribution $D^{\theta}$ is constant, for any nonzero $X \in \Gamma(TM)$. Thus, we get:
\begin{equation}\label{e25}
\cos \theta =\frac{\overline{g}(JP_{1}X, TP_{1}X)}{\|TP_{1}X\| \cdot \|JP_{1}X\|}=\frac{\|TP_{1}X \|}{\|JP_{1}X\|},
\end{equation}
for any nonzero $X \in \Gamma(TM)$.
\end{remark}

\begin{prop}\label{p11}
If $M$ is a hemi-slant submanifold in a metallic Riemannian manifold $(\overline{M},\overline{g},J)$, then:
\begin{equation}\label{e26}
\overline{g}(TP_{1}X,TP_{1}Y)=\cos^2 \theta[p \overline{g}(TP_{1}X,P_{1}Y)+q \overline{g}(P_{1}X,P_{1}Y)]
\end{equation}
and
\begin{equation}\label{e27}
\overline{g}(NX,NY)=\sin^2 \theta[p \overline{g}(TP_{1}X,P_{1}Y)+q \overline{g}(P_{1}X,P_{1}Y)],
\end{equation}
for any $X$, $Y\in \Gamma(TM)$.
\end{prop}
\begin{proof}
Taking $X+Y$ in (\ref{e25}) we have:
$$
\overline{g}(TP_{1}X,TP_{1}Y)=\cos^{2}\theta  \overline{g}(JP_{1}X,JP_{1}Y)= \cos^{2}\theta[p\overline{g}(JP_{1}X,P_{1}Y)+q\overline{g}(P_{1}X,P_{1}Y)],$$
for any $X$, $Y\in \Gamma(TM)$ and using (\ref{e24})(iii) we get (\ref{e26}).

From (\ref{e23}) we get $\overline{g}(TP_{1}X,TP_{1}Y)=\overline{g}(JP_{1}X,JP_{1}Y)-\overline{g}(NX,NY),$ for any $X$, $Y\in \Gamma(TM)$ and it implies (\ref{e27}).
\end{proof}

\begin{remark}
A hemi-slant submanifold $M$ in a Golden Riemannian manifold $(\overline{M},\overline{g},J)$ with the slant angle $\theta$ of the distribution $D^{\theta}$ verifies (\ref{e26}) and (\ref{e27}) with $p=q=1$.
 \end{remark}

\begin{prop}
Let $M$ be a hemi-slant submanifold in a metallic Riemannian manifold $(\overline{M}, \overline{g},J)$ with the slant angle $\theta$ of the distribution $D^{\theta}$. Then:
\begin{equation}\label{e28}
(TP_{1})^2=\cos^2 \theta(p TP_{1}+qI),
\end{equation}
where $I$ is the identity on $\Gamma(D^{\theta})$ and
\begin{equation}\label{e29}
\nabla ((TP_{1})^2)=p \cos^2 \theta \nabla (TP_{1}).
\end{equation}
\end{prop}

\begin{remark}
 Let $M$ be a hemi-slant submanifold in a metallic (or Golden) Riemannian manifold $(\overline{M},\overline{g},J)$, with $TM = D^{\theta}\oplus D^{\perp}$. Then $T(D^{\theta})=D^{\theta}$ and $T(D^{\perp})={0}$.
\end{remark}

\begin{proof}
By using (\ref{e7})(i), we get $\overline{g}(TX,Z)=\overline{g}(X,TZ)=0,$ for any $X \in \Gamma(D^{\theta})$, $Z \in \Gamma(D^{\perp})$.  Thus, $T(D^{\theta}) \perp D^{\perp}$. Since $T(D^{\theta}) \subset \Gamma(TM)$ we obtain that $T(D^{\theta}) \subseteq D^{\theta}$.

Moreover, from  (\ref{e28}) we obtain $X=\frac{1}{q} T(TX-p \cos^2 \theta X)$, for any $X \in \Gamma(D^{\theta})$ (i.e. $P_{1}X=X$), where $(\overline{M},\overline{g},J)$ is a metallic Riemannian manifold.
If $(\overline{M},\overline{g},J)$ is a Golden Riemannian manifold, then $X=T(TX- \cos^2 \theta X)$, for any $X \in \Gamma(D^{\theta})$.
Thus, $D^{\theta} \subseteq T(D^{\theta})$. Since $T(D^{\theta}) \subseteq D^{\theta}$, we get $T(D^{\theta})=D^{\theta}$.

By using (\ref{e24})(ii) we obtain that $D^{\perp}$ is anti-invariant with respect to $J$ and $T(D^{\perp})={0}$.
\end{proof}

\begin{prop}
Let $M$ be an immersed submanifold in a metallic Riemannian manifold $(\overline{M},\overline{g},J)$. Then $M$ is a hemi-slant submanifold in $\overline{M}$ if and only if there exists a constant $\lambda \in [0, 1]$ such that:
 $$D=\{ X \in \Gamma(TM) | T^{2}X= \lambda(pTX+qX)\}$$ is a distribution and $TY=0$, for any $Y$ orthogonal to $D$, $Y\in \Gamma(TM)$, where $p$, $q\in\mathbb{N}^*$.
\end{prop}

\begin{proof}
If $M$ is a hemi-slant submanifold in a metallic Riemannian manifold $(\overline{M},\overline{g},J)$, with $D^{\theta}:=D$ and $TM = D^{\theta}\oplus D^{\perp}$ then, from (\ref{e28}) and $\theta(X) \neq 0$ we have $\lambda=\cos^2 \theta \in [0,1]$.

Conversely, if there exists a real number $\lambda\in[0,1]$ such that $T^2X=\lambda(pTX+qX)$, for any $X \in \Gamma(D)$, it follows that $\cos^2 \theta(X)=\lambda$ which implies that $\theta(X)=\arccos(\sqrt{\lambda})$ does not depend on $X$. If we consider the orthogonal direct sum $TM=D \oplus D^{\bot}$, since $T(D)\subseteq D$ and $TY=0$, for any $Y$ orthogonal to $D$, $Y\in \Gamma(TM)$, we obtain that $M$ is a hemi-slant submanifold in $\overline{M}$ with $D^{\theta}:=D$.
\end{proof}

\begin{remark}
An immersed submanifold $M$ in a Golden Riemannian manifold $(\overline{M},\overline{g},J)$ is a hemi-slant submanifold in $\overline{M}$ if and only if there exists a constant $\lambda \in [0, 1]$ such that $$D=\{ X \in \Gamma(TM) | T^{2}X= \lambda(TX+X)\}$$ is a distribution and $TY=0$, for any $Y \in \Gamma(TM)$ orthogonal to $D$.
\end{remark}

\textbf{Examples 1:}
Let $\mathbb{R}^{4}$ be the Euclidean space endowed with the usual Euclidean metric $<\cdot,\cdot>$.
Let $f: M \rightarrow \mathbb{R}^{4}$ be the immersion given by:
$$f(u,v)=(u \cos t, u \sin t, v,\frac{\sigma}{\sqrt{q}}v),$$
where $M :=\{(u,v) \mid  u>0, t \in (0, \frac{\pi}{2})\}$ and $\sigma:=\sigma_{p,q}=\frac{p+\sqrt{p^{2}+4q}}{2}$ is the metallic number ($p, q \in N^{*}$).

We can find a local orthonormal frame on $TM$ given by:
 $$Z_{1}= \cos t \frac{\partial}{\partial x_{1}} + \sin t \frac{\partial}{\partial x_{2}}, \quad
 Z_{2}=\frac{\partial}{\partial x_{3}} + \frac{\sigma}{\sqrt{q}}\frac{\partial}{\partial x_{4}}.$$

We define the metallic structure $J : \mathbb{R}^{4} \rightarrow \mathbb{R}^{4} $ by:
$$
 J(X_{1},X_{2},X_{3},X_{4})=(\sigma X_{1},\overline{\sigma} X_{2},\sigma X_{3},\overline{\sigma}  X_{4}),
 $$
and we can easily verify that $J^{2}X=p J + q I$ and $<JX, Y> = <X, JY>$, for any $X:=(X_{1},X_{2},X_{3},X_{4})$, $Y:=(Y_{1},Y_{2},Y_{3},Y_{4}) \in \mathbb{R}^{4}$.
 Thus, we obtain:
 $$JZ_{1}= \sigma \cos t \frac{\partial}{\partial x_{1}} + \overline{\sigma} \sin t \frac{\partial}{\partial x_{2}}, \quad
   JZ_{2}= \sigma \frac{\partial}{\partial x_{3}} + \frac{\sigma\overline{\sigma}}{\sqrt{q}}\frac{\partial}{\partial x_{4}}.$$

We remark that $<JZ_{2}, Z_{1}> = <JZ_{2}, Z_{2}>=0$, thus $JZ_{2} \perp span \{Z_{1},Z_{2}\}$.

We find that $\|Z_{1}\|^{2}=1$, $\|JZ_{1}\|^{2}= \sigma^{2} \cos^{2}t +\overline{\sigma}^{2}\sin^{2}t=p (\sigma \cos^{2}t +\overline{\sigma}\sin^{2}t)+q$
and $<JZ_{1}, Z_{1}>=\sigma \cos^{2}t +\overline{\sigma}\sin^{2}t$. Thus, we get
  $$ \cos \theta = \frac{<JZ_{1},Z_{1}>}{\|Z_{1}\|\cdot \|JZ_{1}\|}=\frac{\sigma \cos^{2}t +\overline{\sigma}\sin^{2}t}{\sqrt{\sigma^{2} \cos^{2}t +\overline{\sigma}^{2}\sin^{2}t}}=
 \frac{\sigma \cos^{2}t +\overline{\sigma}\sin^{2}t}{\sqrt{p (\sigma \cos^{2}t +\overline{\sigma}\sin^{2}t)+q}}.$$

 In particular, for $t=\frac{\pi}{4}$ we get $\cos \theta = \frac{\sigma +\overline{\sigma}}{\sqrt{\sigma^{2} +\overline{\sigma}^{2}}}$.

We define the distributions $D^{\theta}=span\{Z_{1}\}$ and $D^{\perp}=span\{Z_{2}\}$. We have $J(D^{\perp})\subset \Gamma(T^{\perp}M)$ and
$D^{\theta}$ is a slant distribution, with the slant angle $\theta = \arccos \frac{\sigma \cos^{2}t +\overline{\sigma}\sin^{2}t}{\sqrt{p (\sigma \cos^{2}t +\overline{\sigma}\sin^{2}t)+q}}.$
The Riemannian metric tensor of $D^{\theta} \oplus D^{\perp}$ is given by
$
g=du^{2} + \frac{p\sigma+2q}{q}d v^{2}.
$
Thus, $M$ is a hemi-slant submanifold in the metallic Riemannian manifold  $(\mathbb{R}^{4}, <\cdot,\cdot>, J)$, with $TM=D^{\theta} \oplus D^{\perp}$.

In particular, for $p=q=1$ and $\phi:=\sigma_{1,1}=\frac{1+\sqrt{5}}{2}$ is the Golden number ($\overline{\phi}:=1-\phi$), the immersion $f: M \rightarrow \mathbb{R}^{4}$ is given by $f(u,v)=(u \cos t, u \sin t, v,\phi v)$ and the Golden structure $J : \mathbb{R}^{4} \rightarrow \mathbb{R}^{4} $ can be defined by
$$ J(X_{1},X_{2},X_{3},X_{4})=(\phi X_{1},\overline{\phi} X_{2},\phi X_{3},\overline{\phi}  X_{4}).$$ 

The distribution $D^{\perp}=span\{Z_{2}\}$ verifies $J(D^{\perp})\subset \Gamma(T^{\perp}M)$ and $D^{\theta}=span\{Z_{1}\}$ is a slant distribution, with the slant angle $\theta = \arccos \frac{\phi \cos^{2}t +\overline{\phi}\sin^{2}t}{ \sqrt{(\phi \cos^{2}t +\overline{\phi}\sin^{2}t)+1}}$
and the Riemannian metric tensor of $D^{\theta} \oplus D^{\perp}$ is given by $g=du^{2} + (\phi+2)d v^{2}.$
Thus, $TM=D^{\theta} \oplus D^{\perp}$ and $M$ is a hemi-slant submanifold in the Golden Riemannian manifold  $(\mathbb{R}^{4}, <\cdot,\cdot>, J)$.

If we consider the metallic structure $\overline{J} : \mathbb{R}^{4} \rightarrow \mathbb{R}^{4} $ given by
$$ \overline{J}(X_{1},X_{2},X_{3},X_{4})=(\sigma X_{1},\sigma X_{2},\sigma X_{3},\overline{\sigma}  X_{4}),$$
then we obtain:
$\overline{J}Z_{1}= \sigma Z_{1}$ and $\overline{J}Z_{2}= \sigma \frac{\partial}{\partial x_{3}} + \frac{\sigma\overline{\sigma}}{\sqrt{q}}\frac{\partial}{\partial x_{4}}.$
In this case we obtain the distributions $D^{\perp}=span\{Z_{2}\}$ and
$D^{\theta}=span\{Z_{1}\}$ with the slant angle $\theta = \arccos 1 = 0.$ Thus, $TM=D^{\theta} \oplus D^{\perp}$ and $M$ is a semi-invariant submanifold in the metallic Riemannian manifold  $(\mathbb{R}^{4}, <\cdot,\cdot>, \overline{J})$. Similarly, for $p=q=1$ we obtain that $M$ is a semi-invariant submanifold in the Golden Riemannian manifold  $(\mathbb{R}^{4}, <\cdot,\cdot>, \overline{J})$.

\textbf{Examples 2:}
Let $\mathbb{R}^{7}$ be the Euclidean space endowed with the usual Euclidean metric $<\cdot,\cdot>$.
Let $f: M \rightarrow \mathbb{R}^{7}$ be the immersion given by:
$$f(u,v,w)=(\frac{1}{\sqrt{3}} u \cos t,\frac{1}{\sqrt{3}} u \sin t, v,\frac{\sigma}{\sqrt{q}}v, \frac{\sqrt{q}}{\sigma}w, w, \frac{\sqrt{2}}{\sqrt{3}}u),$$
where $M :=\{(u,v,w) \mid u>0, t \in (0, \frac{\pi}{2})\}$ and $\sigma:=\sigma_{p,q}$ is the metallic number ($p, q \in N^{*}$).

We can find a local orthonormal frame on $TM$ given by:
 $$Z_{1}= \frac{1}{\sqrt{3}}\cos t \frac{\partial}{\partial x_{1}} + \frac{1}{\sqrt{3}}\sin t \frac{\partial}{\partial x_{2}}+\frac{\sqrt{2}}{\sqrt{3}}\frac{\partial}{\partial x_{7}},
 \quad Z_{2}=\frac{\partial}{\partial x_{3}} + \frac{\sigma}{\sqrt{q}}\frac{\partial}{\partial x_{4}},
 \quad
 Z_{3}=\frac{\sqrt{q}}{\sigma} \frac{\partial}{\partial x_{5}} + \frac{\partial}{\partial x_{6}}.$$

We define the metallic structure $J : \mathbb{R}^{7} \rightarrow \mathbb{R}^{7} $ by:
$$
 J(X_{1},X_{2},X_{3},X_{4},X_{5},X_{6},X_{7})=(\sigma X_{1},\overline{\sigma} X_{2},\sigma X_{3},\overline{\sigma}  X_{4},\sigma X_{5},\overline{\sigma} X_{6},\sigma X_{7})
 $$
 and we can easily verify that $J^{2}X=p J + q I$ and $<JX, Y> = <X, JY>$, for any $X:=(X_{1},X_{2},X_{3},X_{4},X_{5},X_{6},X_{7})$, $Y:=(Y_{1},Y_{2},Y_{3},Y_{4},Y_{5},Y_{6},Y_{7}) \in \mathbb{R}^{7}$.
 Thus, we obtain:
 $$JZ_{1}= \frac{1}{\sqrt{3}}\sigma \cos t \frac{\partial}{\partial x_{1}} + \frac{1}{\sqrt{3}}\overline{\sigma} \sin t \frac{\partial}{\partial x_{2}}+\frac{\sqrt{2}}{\sqrt{3}}\sigma\frac{\partial}{\partial x_{7}},$$
   $$ JZ_{2}= \sigma \frac{\partial}{\partial x_{3}} - \sqrt{q}\frac{\partial}{\partial x_{4}}, \quad
   JZ_{3}= \sqrt{q} \frac{\partial}{\partial x_{5}} + \overline{\sigma}\frac{\partial}{\partial x_{6}}.$$

We find that $JZ_{2} \perp span \{Z_{1},Z_{2},Z_{3}\}$, $JZ_{3} \perp span \{Z_{1},Z_{2},Z_{3}\}$.
Moreover, we have $\|Z_{1}\|^{2}=1$, $\|Z_{2}\|^{2}=\frac{p\sigma+2q}{q}$ and $\|Z_{3}\|^{2}=\frac{p\sigma+2q}{p\sigma+q}$.

Thus, we get
  $$ \cos \theta = \frac{<JZ_{1},Z_{1}>}{\|Z_{1}\|\cdot \|JZ_{1}\|}=\frac{\sigma (\cos^{2}t+2) +\overline{\sigma}\sin^{2}t}{\sqrt{3[\sigma^{2} (\cos^{2}t+2) +\overline{\sigma}^{2}\sin^{2}t]}}.$$

In particular, for $t=\frac{\pi}{4}$ we get $\cos \theta = \frac{5\sigma +\overline{\sigma}}{\sqrt{3(5\sigma^{2} +\overline{\sigma}^{2})}}$.

We define the distributions $D^{\theta}=span\{Z_{1}\}$ and $D^{\perp}=span\{Z_{2},Z_{3}\}$. We have $J(D^{\perp})\subset \Gamma(T^{\perp}M)$ and
$D^{\theta}$ is a slant distribution, with the slant angle
$\theta = \arccos \frac{\sigma (\cos^{2}t+2) +\overline{\sigma}\sin^{2}t}{\sqrt{3[\sigma^{2} (\cos^{2}t+2) +\overline{\sigma}^{2}\sin^{2}]}}.$
The Riemannian metric tensor of $D^{\theta} \oplus D^{\perp}$ is given by
$
g=du^{2} + \frac{p\sigma+2q}{q}d v^{2}+\frac{p\sigma+2q}{p\sigma+q}d w^{2}.
$
Thus, $TM=D^{\theta} \oplus D^{\perp}$ and $M$ is a hemi-slant submanifold in the metallic Riemannian manifold  $(\mathbb{R}^{7}, <\cdot,\cdot>, J)$.

In particular, for $p=q=1$ and $\phi:=\sigma_{1,1}$ is the Golden number ($\overline{\phi}:=1-\phi$), the immersion $f: M \rightarrow \mathbb{R}^{7}$ is given by $$f(u,v,w)=(\frac{1}{\sqrt{3}} u \cos t,\frac{1}{\sqrt{3}} u \sin t, v, \phi v, \overline{\phi}w, w, \frac{\sqrt{2}}{\sqrt{3}}u),$$ and the Golden structure $J : \mathbb{R}^{7} \rightarrow \mathbb{R}^{7}$  can be defined by
by $$ J(X_{1},X_{2},X_{3},X_{4},X_{5},X_{6},X_{7})=(\phi X_{1},\overline{\phi} X_{2},\phi X_{3},\overline{\phi}  X_{4},\phi X_{5},\overline{\phi} X_{6},\phi X_{7}).$$
The distributions $D^{\perp}=span\{Z_{2},Z_{3}\}$ verifies $J(D^{\perp})\subset \Gamma(T^{\perp}M)$ and slant distribution is $D^{\theta}=span\{Z_{1}\}$, with the slant angle
$\theta = \arccos \frac{\phi (\cos^{2}t+2) +\overline{\phi}\sin^{2}t}{\sqrt{3[\phi^{2} (\cos^{2}t+2) +\overline{\phi}^{2}\sin^{2}]}}.$
The Riemannian metric tensor of $D^{\theta} \oplus D^{\perp}$ is given by
$g=du^{2} + (\phi+2)d v^{2}+\frac{\phi+2}{\phi+1}d w^{2}.$
Thus, $TM=D^{\theta} \oplus D^{\perp}$ and $M$ is a hemi-slant submanifold in the metallic Riemannian manifold  $(\mathbb{R}^{7}, <\cdot,\cdot>, J)$.

If we consider the metallic structure $\overline{J} :\mathbb{R}^{7} \rightarrow \mathbb{R}^{7}$  defined by
by  
$$\overline{J}(X_{1},X_{2},X_{3},X_{4},X_{5},X_{6},X_{7})=(\sigma X_{1},\sigma X_{2},\sigma X_{3},\overline{\sigma}  X_{4},\sigma X_{5},\overline{\sigma} X_{6},\sigma X_{7}),$$
then we obtain:
$\overline{J}Z_{1}= \sigma Z_{1}$, $\overline{J}Z_{2}= \sigma \frac{\partial}{\partial x_{3}} - \sqrt{q}\frac{\partial}{\partial x_{4}}$ and 
   $\overline{J}Z_{3}= \sqrt{q} \frac{\partial}{\partial x_{5}} + \overline{\sigma}\frac{\partial}{\partial x_{6}}.$
In this case we obtain the distributions $D^{\perp}=span\{Z_{2},Z_{3}\}$ and $D^{\theta}=span\{Z_{1}\}$ with the slant angle $\theta = \arccos 1 = 0.$ Thus, $TM=D^{\theta} \oplus D^{\perp}$ and $M$ is a semi-invariant submanifold in the metallic Riemannian manifold  $(\mathbb{R}^{7}, <\cdot,\cdot>, \overline{J})$. Similarly, for $p=q=1$ we obtain that $M$ is a semi-invariant submanifold in the Golden Riemannian manifold  $(\mathbb{R}^{7}, <\cdot,\cdot>, \overline{J})$.

\section{On the integrability of the distributions of a hemi-slant submanifold}

In this section we investigate the conditions for the integrability of the distributions of a hemi-slant submanifold in a metallic (or Golden) Riemannian manifold.

\begin{prop}
If $M$ is a hemi-slant submanifold in a locally metallic (or locally Golden) Riemannian manifold $(\overline{M},\overline{g},J)$, then 
\begin{equation}\label{e30}
\nabla_{X}TY-\nabla_{Y}TX-A_{NY}X+A_{NX}Y \in \Gamma(D^{\theta}),
 \end{equation}
 for any $X,Y \in \Gamma(D^{\theta})$.
\end{prop}
\begin{proof}
By using (\ref{e7})(i), we obtain:
$\overline{g}(T([X,Y]),Z)=\overline{g}([X,Y],TZ)=0,$ for any $X,Y \in \Gamma(D^{\theta})$ and $Z \in \Gamma(D^{\perp})$ (i.e. $TZ=0$). Thus, $T([X,Y]) \in \Gamma(D^{\theta})$ and using (\ref{e20}) we obtain (\ref{e30}).
\end{proof}

\begin{prop}
If $M$ is a hemi-slant submanifold in a locally metallic (or locally Golden) Riemannian manifold $(\overline{M},\overline{g},J)$, then the distribution $D^{\theta}$ is integrable.
\end{prop}
\begin{proof}
By using (\ref{e3}), we have
$\overline{g}(\overline{\nabla}_{X}Y,Z) = \frac{1}{q} [\overline{g}(J\overline{\nabla}_{X}Y,JZ) - p \overline{g}(\overline{\nabla}_{X}Y,JZ)],$
for any $X,Y \in \Gamma(D^{\theta})$, $Z \in \Gamma(D^{\perp})$.

From $\overline{\nabla}J=0$ we get $J \overline{\nabla}_{X}Y = \overline{\nabla}_{X}JY$ and using $JZ=NZ$, for any $Z \in \Gamma(D^{\perp})$, we obtain
$q \overline{g}(\overline{\nabla}_{X}Y,Z) = \overline{g}(\overline{\nabla}_{X}JY,NZ) - p \overline{g}(\overline{\nabla}_{X}Y,NZ).$
Thus, from (\ref{e11}) and (\ref{e12}) we get
$q\overline{g}(\overline{\nabla}_{X}Y,Z) =\overline{g}(h(X,TY),NZ)+\overline{g}(\nabla_{X}^{\perp}NY,NZ) - p \overline{g}(h(X,Y),NZ).$
From (\ref{e13})(ii) and (\ref{e17})(ii) we obtain
$
\nabla_{X}^{\perp}NY =n h(X,Y) - h(X,TY) + N \nabla_{X}Y,
$
for any $X,Y \in \Gamma(D^{\theta})$.
 Thus, we get
 $$q\overline{g}(\overline{\nabla}_{X}Y,Z) =\overline{g}(nh(X,Y),NZ)+\overline{g}(N\nabla_{X}Y,NZ) - p \overline{g}(h(X,Y),NZ),$$
 which implies
$$q \overline{g}([X,Y],Z)= \overline{g}(N \nabla_{X}Y,NZ) - \overline{g}(N \nabla_{Y}X,NZ)=\overline{g}(N[X,Y],NZ),$$ 
for any $X,Y \in \Gamma(D^{\theta})$ and $Z \in \Gamma(D^{\perp})$.
Thus, from (\ref{e27}) and (\ref{e7})(i) we have
$$q \overline{g}([X,Y],Z)= \sin^{2} \theta [p \overline{g}(P1[X,Y],TP_{1}Z)+ q \overline{g}(P1[X,Y],P_{1}Z)].$$ 
By using $P_{1}Z=0$ for any $Z \in \Gamma(D^{\perp})$ (where $P_{1}Z$ is the projection of $Z$ on $\Gamma(D^{\theta})$), we obtain $\overline{g}([X,Y],Z)=0$, for any $X,Y \in \Gamma(D^{\theta})$, $Z \in \Gamma(D^{\perp})$ which implies that $[X,Y] \in \Gamma(D^{\theta})$.
\end{proof}

\begin{prop}
Let $M$ be a hemi-slant submanifold in a locally metallic (or locally Golden) Riemannian manifold $(\overline{M},\overline{g},J)$. Then the distribution $D^{\perp}$ is integrable if and only if
\begin{equation}\label{e31}
A_{NZ}W=0,
\end{equation}
 for any $Z,W\in \Gamma(D^{\perp})$.
\end{prop}
\begin{proof}
If $M$ is a hemi-slant submanifold in a locally metallic (or locally Golden) Riemannian manifold $(\overline{M},\overline{g},J)$ then, for any $Z,W \in \Gamma(D^{\perp})$ we have $TZ=TW=0$ which implies $\nabla_{Z}TW=\nabla_{W}ZX=0$. By using (\ref{e24})(ii) and (\ref{e20}) we get $T([Z,W])=0$ if and only if $A_{NZ}W=A_{NW}Z$ holds, for any $Z,W \in \Gamma(D^{\perp})$.

From (\ref{e17})(i), for any $X\in \Gamma(TM)$ and $Z,W \in \Gamma(D^{\perp})$, we get
$$ \overline{g}(A_{NZ}X,W)+\overline{g}(th(X,Z),W)=\overline{g}((\nabla_{X}T)Z,W)=-\overline{g}(\nabla_{X}Z,TW)=0,$$
which implies $ \overline{g}(A_{NZ}X,W) = -\overline{g}(th(X,Z),W)$.
From 
$$\overline{g}(A_{NZ}X,W)=\overline{g}(A_{NZ}W,X)=\overline{g}(A_{NW}Z,X)=\overline{g}(h(X,Z),NW)=\overline{g}(th(X,Z),W),$$ we obtain $\overline{g}(A_{NZ}W,X)=0$ for any $X\in \Gamma(TM)$ and $Z,W \in \Gamma(D^{\perp})$. Thus, $A_{NZ}W=0$, for any $Z,W \in \Gamma(D^{\perp})$.

Conversely, if $A_{NZ}W=0$, for any $Z,W \in \Gamma(D^{\perp})$ then, from $\overline{g}(th(X,Z),W)=\overline{g}(h(X,Z),NW)=\overline{g}(A_{NW}Z,X)=0$ and (\ref{e17})(i), we get
$$0=\overline{g}((\nabla_{Z}T)W,X)=\overline{g}(T\nabla_{Z}W,X)=\overline{g}(\nabla_{Z}W,TX),$$ for any $Z,W \in \Gamma(D^{\perp})$, $X \in \Gamma(D^{\theta})$. From $T(D^{\theta})=D^{\theta}$, we obtain $\nabla_{Z}W \in \Gamma(D^{\perp})$ which implies $[Z,W] \in \Gamma(D^{\perp})$.
\end{proof}

\begin{prop}
Let $M$ be a hemi-slant submanifold in a locally metallic (or locally Golden) Riemannian manifold $(\overline{M},\overline{g},J)$. Then, the anti-invariant distribution  $D^{\perp}$ is integrable if and only if
\begin{equation}\label{e32}
(\nabla_{Z}T)W=(\nabla_{W}T)Z
\end{equation}
 for any $Z,W \in \Gamma(D^{\perp})$.
\end{prop}
\begin{proof}
By using (\ref{e17}) we get $(\nabla_{Z}T)W - (\nabla_{W}T)Z = A_{NW}Z - A_{NZ}W$,  for any $Z,W \in \Gamma(D^{\perp})$ and using (\ref{e31}) we obtain the conclusion.
\end{proof}

\begin{remark}
Let $M$ be a hemi-slant submanifold in a locally metallic (or locally Golden) Riemannian manifold ($\overline{M},\overline{g},J)$. If $(\nabla_{Z}T)W =0$, for any $Z,W \in \Gamma(D^{\perp})$, then $D^{\perp}$ is integrable.
\end{remark}

\begin{prop}
Let $M$ be a hemi-slant submanifold in a locally metallic (or locally Golden) Riemannian manifold ($\overline{M},\overline{g},J)$. If $(\overline{\nabla}_{X}N)Y =0$, for any $X,Y \in \Gamma(D^{\theta})$ then, either $M$ is a $D^{\theta}$ geodesic submanifold (i.e $h(X,Y)=0$) or $h(X,Y)$ is an eigenvector of $n$, with eigenvalues
\begin{equation}\label{e35}
\lambda_{1}=\frac{p \cos^{2}\theta + \cos \theta \sqrt{p^{2}\cos^{2}\theta + 4q}}{2}, \quad \lambda_{2}=\frac{p \cos^{2}\theta - \cos \theta \sqrt{p^{2}\cos^{2}\theta + 4q}}{2}.
\end{equation}
\end{prop}

\begin{proof}
By using $(\overline{\nabla}_{X}N)Y =0$ for any $X,Y \in \Gamma(D^{\theta})$ and (\ref{e17})(ii) we obtain $nh(X,Y)=h(X,TY)$. From (\ref{e28}) we get, for any $X,Y \in \Gamma(D^{\theta})$:
$$ n^{2}h(X,Y)=h(X,T^{2}Y)=p\cos^{2}\theta nh(X,Y) + q\cos^{2}\theta h(X,Y).$$ Thus, we obtain either $M$ is a $D^{\theta}$ geodesic submanifold or $h(X,Y)$ is an eigenvector of $n$ with eigenvalue $\lambda$, which verifies the equation $\lambda^{2}-p\cos^{2}\theta \lambda - q\cos^{2}\theta =0$ and (\ref{e35}) holds.
\end{proof}

\section{Mixed totally geodesic hemi-slant submanifolds}
In the next propositions, we consider hemi-slant submanifolds in a locally metallic (or locally Golden) Riemannian manifold and we find some conditions for these submanifolds to be $D^{\theta} - D^{\perp}$ mixed totally geodesic (i.e. $h(X,Y)=0$, for any $X \in \Gamma(D^{\theta})$ and $Y \in \Gamma(D^{\perp})$).

\begin{prop}
If $M$ is a hemi-slant submanifold in a locally metallic (or locally Golden) Riemannian manifold $(\overline{M},\overline{g},J)$, then $M$ is a $D^{\theta}-D^{\perp}$ mixed totally geodesic submanifold if and only if $A_{V}X \in \Gamma(D^{\theta})$ and  $A_{V}Y \in \Gamma(D^{\perp})$, for any $X \in \Gamma(D^{\theta})$, $Y \in \Gamma(D^{\perp})$ and $V \in \Gamma(T^{\bot}M)$.
\end{prop}

\begin{proof}
From $\overline{g}(A_{V}X,Y)=\overline{g}(A_{V}Y,X)=\overline{g}(h(X,Y),V)$, for any $X \in \Gamma(D^{\theta}), Y \in \Gamma(D^{\perp})$ and $V \in \Gamma(T^{\bot}M)$ we obtain that $M$ is a $D^{\theta}-D^{\perp}$ mixed totally geodesic submanifold in the locally metallic (or locally Golden) Riemannian manifold if and only if $A_{V}X \in \Gamma(D^{\theta})$ and  $A_{V}Y \in \Gamma(D^{\perp})$, for any $X \in \Gamma(D^{\theta})$, $Y \in \Gamma(D^{\perp})$ and $V \in \Gamma(T^{\bot}M)$.
\end{proof}

\begin{prop}
Let $M$ be a proper hemi-slant submanifold in a locally metallic (or locally Golden) Riemannian manifold ($\overline{M},\overline{g},J)$. If $(\overline{\nabla}_{X}N)Z =0$, for any $X\in \Gamma(TM)$ and $Z \in \Gamma(D^{\perp})$, then $M$ is a $D^{\theta}-D^{\perp}$ mixed totally geodesic submanifold in $\overline{M}$ .
\end{prop}

\begin{proof}
If $X \in \Gamma(D^{\theta})$ and $Z \in \Gamma(D^{\perp})$ then, from $(\overline{\nabla}_{X}N)Z =0$,  (\ref{e17})(ii) and $TZ=0$ we get $h(Z,TX)=nh(X,Z)=h(X,TZ)=0$.
From (\ref{e28}), we have
$$ 0=n^{2}h(Z,X)=h(Z,T^{2}X)=p\cos^{2}\theta nh(Z,TX) + q\cos^{2}\theta h(Z,X)$$ and we obtain $q\cos^{2}\theta h(Z,X)=0$. By using $\theta \neq \frac{\pi}{2}$ and $q\neq 0$,
we get $h(X,Z)=0$, for any $X \in \Gamma(D^{\theta})$ and $Z \in \Gamma(D^{\perp})$.
\end{proof}

\linespread{1}
Cristina E. Hretcanu, \\ Stefan cel Mare University of Suceava, Romania, e-mail: criselenab@yahoo.com\\
Adara M. Blaga, \\ West University of Timisoara, Romania, e-mail: adarablaga@yahoo.com.

\end{document}